\documentclass[12pt,amsfonts]{article}
\usepackage[margin=1in]{geometry}  
\usepackage{graphicx}              
\usepackage[latin1]{inputenc}
\usepackage{amsmath, amssymb, amsthm, epsfig}               
\usepackage{enumerate}
\usepackage{amsfonts}              
\usepackage{amsthm}                
\usepackage{amsthm}
\usepackage{enumerate}
\theoremstyle{plain}
\usepackage{color}

\usepackage{tikz}  
\usepackage{pgfplots}

\def\nd{\noindent}
\def\thend{\rule{3mm}{3mm}}

\newtheorem{claim}{Claim}[section]
\newtheorem{theorem}{Theorem}[section]

\newtheorem{proposition}{Proposition}[section]
\newtheorem{lemma}{Lemma}[section]

\newtheorem{definition}{Definition}[section]
\newtheorem{corollary}{Corollary}[section]
\newtheorem*{theorem*}{Theorem}

\numberwithin{equation}{section}

\begin{document}
\title{Multiple normalized solutions to a logarithmic Schr\"{o}dinger equation via Lusternik-Schnirelmann category}
\author{ Claudianor O. Alves,\footnote{C.O. Alves was partially supported by CNPq/Brazil   grant 304804/2017-7.} \,\, Chao Ji\footnote{Corresponding author} \footnote{C. Ji was partially supported by National Natural Science Foundation of China (No. 12171152).} }

\maketitle

\begin{abstract}
In this paper  our objective is to investigate the existence of multiple normalized solutions to the logarithmic Schr\"{o}dinger equation given by
\begin{align*}
 \left\{
\begin{aligned}
&-\epsilon^2 \Delta u+V( x)u=\lambda u+u \log u^2, \quad
\quad
\hbox{in }\mathbb{R}^N,\\
&\int_{\mathbb{R}^{N}}|u|^{2}dx=a^{2}\epsilon^N,
\end{aligned}
\right.
\end{align*}
where $a, \epsilon>0, \lambda \in \mathbb{R}$ is an unknown parameter that appears as a Lagrange multiplier and $V: \mathbb{R}^N \rightarrow[-1, \infty)$ is a continuous function. Our analysis demonstrates that the number of normalized solutions of the equation is associated with the topology of the set where the potential function $V$ attains its minimum value. To prove the main result, we employ minimization techniques and use the Lusternik-Schnirelmann category. Additionally, we introduce a new function space where the energy functional associated with the problem is of class $C^1$.

\end{abstract}



\section{Introduction}

In recent decades, there has been significant interest in studying the following nonlinear Schr\"{o}dinger equation in $\mathbb{R}^N$
\begin{equation}\label{SCH}
	i\frac{\partial \psi}{\partial t}+\triangle \psi-V(x)\psi+h(|\psi|^{2})\psi=0
	\quad
	\hbox{in }\mathbb{R}^N
\end{equation}
where $\psi=\psi(t,x)$, $V$ is the potential, and $h$ is a function. One important aspect related to equation \eqref{SCH} is the search for its standing wave solutions. A standing wave is a solution of the form $\psi(t, x)=e^{-i\lambda t}u(x)$, where $\lambda\in \mathbb{R}$ and $u:\mathbb{R}^N\rightarrow \mathbb{R}$ is a time-independent function that satisfies the following nonlinear elliptic equation
\begin{equation} \label{g00}
-\Delta u+V(x)u=\lambda u+g(u), \quad \text{in} \quad \mathbb{R}^N,
\end{equation}
where $g(u)=h(|u|^2)u$.\\

There are two aspects to the study of standing waves of the Schr\"{o}dinger equation \eqref{SCH}.  On one hand, one can choose the fixed frequency $\lambda\in \mathbb{R}$ and investigate the existence of nontrivial solutions for (\ref{g00}) obtained as critical points of the energy functional $J_{\lambda}:H^{1}(\mathbb{R}^N) \to \mathbb{R}$ given by
$$
J_\lambda(u)=\frac{1}{2}\int_{\mathbb{R}^N}\Big(|\nabla u|^2+(V(x)-\lambda) |u|^2\Big)\,dx-\int_{\mathbb{R}^N}G(u)\,dx,
$$
where $G(t)=\int_{0}^{t}g(s)\,ds$. When $g(u)=u\log u^2$, the equation \eqref{g00} has significant physical applications in quantum mechanics, quantum optics, nuclear physics, transport and diffusion phenomena, open quantum systems, effective quantum gravity, theory of superfluidity, and Bose-Einstein condensation (see \cite{z} and the references therein). In addition to its applications, equation \eqref{g00} presents challenging mathematical problems. For instance, the associated energy functional is not $C^{1}$ smooth; in fact, it is not even continuous, but only lower semicontinuous, due to the existence of $u \in H^{1}(\mathbb{R}^N)$ such that $\int_{\mathbb{R}^N}u^{2}\log u^2 , dx=-\infty$.  Indeed, one can consider a smooth function that satisfies
 $$
 u(x)=
 \left\{
 \begin{array}{l}
 (|x|^{N/2}\log(|x|))^{-1}, \quad |x| \geq 3, \\
 0, \quad\quad\quad\quad\quad\quad \quad\quad\,|x| \leq 2.	
 \end{array}
 \right.
 $$
 Various techniques have been employed by different authors to address the existence, multiplicity, and concentration of solutions under assumptions on the potential $V$ (see \cite{Ad, AlvesdeMorais, AlvesdeMoraisFigueiredo, AlvesChaobefore, AlvesChao, AlvesChao1, AlvesChao3, Cazenave, DMS, sz, sz2, WZ, ZW2020} and the references therein).\\

The pioneering work we would like to mention is by Cazenave \cite{Cazenave}. In \cite{Cazenave}, the author investigated the logarithmic Schr\"{o}dinger equation given by
$$
i u_t+\Delta u+u \log u^2=0, \quad(t, x) \in \mathbb{R} \times \mathbb{R}^N,
$$
on the space $W:=\left\{u \in H^1\left(\mathbb{R}^N\right) ; \int_{\mathbb{R}^N}\left|u^2 \log u^2\right| d x<\infty\right\}$ and considered a suitable Luxemburg norm type on $W$. More precisely, by taking the $\mathrm{N}$-function
$$
\begin{gathered}
A(s):=\left\{\begin{array}{c}
-\frac{1}{2} s^2 \log s^2, \quad \,\, \,\,\,0 \leq s \leq e^{-3}, \\
3 s^2+4 e^{-3} s-e^{-6}, \quad s \geq e^{-3},
\end{array}\right. \\
\|u\|_A=\inf \left\{\lambda>0 ; \int_{\Omega} A\left(\frac{|u|}{\lambda}\right) \leq 1\right\},
\end{gathered}
$$
and $\left(\|\cdot\|_{H^1\left(\mathbb{R}^N\right)}+\|\cdot\|_A\right)$ as the norm on $W$, the author showed the existence of infinitely many critical points for the functional
$$
L(u)=\frac{1}{2} \int_{\mathbb{R}^N}|\nabla u|^2 d x-\frac{1}{2} \int_{\mathbb{R}^N} u^2 \log u^2, \quad u \in W
$$
on the set $\Sigma:=\left\{u \in W ; \int_{\mathbb{R}^N}|u|^2 d x=1\right\}$.

Subsequently, additional frameworks were developed to investigate the logarithmic Schr\"{o}dinger equations. In their work \cite{DMS}, d'Avenia, Montefusco, and Squassina studied the existence, uniqueness, and multiplicity of solutions for the following equation using the critical point theory for nonsmooth functionals developed by Degiovanni and Zani in \cite{DZ}
$$
\left\{\begin{aligned}
&-\Delta u+u=u \log u^2, \text { in } \mathbb{R}^N, \\
& u \in H^1\left(\mathbb{R}^N\right).
\end{aligned}\right.
$$

In recent papers \cite{sz, sz2}, Squassina and Szulkin decomposed the functional associated with the logarithmic Schr\"{o}dinger equation into a sum of a $C^{1}$ functional and a convex lower semicontinuous functional. They investigated the existence of ground states and multiple solutions for logarithmic Schr"{o}dinger equations with periodic potential. Specifically, the authors used a decomposition of the form
\begin{equation} \label{decomposition}
F_2(s)-F_1(s)=\frac{1}{2}s^{2}\log s^{2}, \quad s\in \mathbb{R},
\end{equation}
where $F_1$ is a $C^1$ and nonnegative convex function, and $F_2$ is also a $C^1$ function with subcritical growth (see Section 2  in the subsequent discussion). Motivated by \cite{sz}, Alves and de Morais Filho \cite{AlvesdeMorais} investigated the existence and concentration of positive solutions to the following problem
$$
-{\epsilon}^2\Delta u+ V(x)u=u \log u^2, \,\, \mbox{in} \, \mathbb{R}^{N},
\eqno{(P_\epsilon)}
$$
where $\epsilon>0$. They required $V$ to satisfy a global assumption
\begin{equation}\label{ps1}
V_{\infty}:=\lim_{\vert x\vert\rightarrow \infty}\, V(x)> \inf_{x\in \mathbb{R}^N}\, V(x)=V_{0}>-1.
\end{equation}
Subsequently, Alves and Ji   \cite{AlvesChao} considered the multiple positive solutions to  $(P_\epsilon)$  under the same assumption (\ref{ps1}). Specifically, they demonstrated that the "shape" of the graph of the function $V$ influences the number of nontrivial solutions. Furthermore, under a local assumption on the potential $V$, Alves and Ji \cite{AlvesChaobefore} employed the penalization method to establish the existence of single-peak solutions for $(P_\epsilon)$.\\

Recently, Wang and Zhang \cite{WZ} developed a new method to study the logarithmic Schr\"{o}dinger equation, namely, power approximations. They uncovered a relationship between power-law nonlinear scalar field equations and logarithmic-law scalar field equations. Furthermore, they demonstrated that as $p\downarrow 2$ in the power-law scalar field equations, the ground state solutions converge to the ground state solutions of the logarithmic-law equations. Additionally, Zhang and Wang \cite{ZW2020} explored the concentration of nodal solutions for logarithmic scalar field equations by applying the ideas and methods presented in their paper \cite{WZ}.

Very recently,  we have observed that Alves and da Silva \cite{Ad} established the existence and multiplicity of positive solutions for logarithmic Schr\"{o}dinger equations by introducing a new function space where the energy functional is $C^{1}$. They used the Lusternik-Schnirelmann category theory in their study.

On the other hand, one can consider the case  $\lambda \in \mathbb{R}$ is unknown. In this case,  $\lambda\in \mathbb{R}$  appears as a Lagrange multiplier and the $L^{2}$-norm of solutions prescribed. From the physical point of view, this line of research holds particular significance as it accounts for the conservation of mass. Additionally, it provides valuable insights into the dynamic properties of standing waves in \eqref{SCH}, such as stability or instability \cite{BC, CL1}. In this paper, we focus primarily on the aspect.

If the potential $V\equiv 0$, a vast literature is dedicated to the following problem
\begin{align*}\label{problemnew}
 \left\{
\begin{aligned}
&-\Delta u=\lambda u+g(u), \quad
\quad
\hbox{in }\mathbb{R}^N,\\
&\int_{\mathbb{R}^{N}}|u|^{2}dx=a^{2},
\end{aligned}
\right.
\end{align*}
we refer to \cite{AJ1, CCM, AT, CL1, IM, IM1, jeanjean1, JL, Sh, Shu, Nicola1, Nicola2} and the references therein. In particular, in our scope, we would like to mention  the attainability of the $L^{2}$-constraint minimization problem
$$
E_{\alpha}=\inf\Big\{\frac{1}{2}\int_{\mathbb{R}^N}|\nabla u|^2 \,dx-\int_{\mathbb{R}^N} G(\vert u\vert) \,dx: u\in H^{1}(\mathbb{R}^N), \int_{\mathbb{R}^{N}}|u|^{2}dx=\alpha^{2}\Big\}.
$$
For the case $g(u)=\vert u\vert^{p-1}u \,\, (1<p<1+\frac{4}{N})$, by assuming $H^{1}$-precompactness of any minimizing sequences, Cazenave and Lions \cite{CL1} showed the attainability of the $L^{2}$-constraint minimization problem and orbital stability of the set of global minimizers. For this purpose, the strict subadditivity condition
\begin{equation}\label{Subadd}
E_{\alpha+\beta}<E_{\alpha}+E_{\beta}
\end{equation}
a crucial role in this context. This condition excludes the dichotomy of minimizing
sequences, and implies $H^{1}$-precompactness. However, establishing the validity of \eqref{Subadd} is challenging for general functions $g$. Shibata \cite{Sh} proved the subadditivity condition \eqref{Subadd} using a scaling argument. Furthermore, the author demonstrated that there exists $\alpha_{0}\geq 0$ such that a global minimizer exists for $\alpha>\alpha_{0}$, while no global minimizer exists for $\alpha<\alpha_{0}$. Recently, Jeanjean and Lu \cite{JL} extended these results to more general nonlinearities.

If the potential $V\not\equiv 0$, then the scaled function $u(\kappa x)$ in \cite{Sh} becomes less effective, making the $L^{2}$-constraint minimization problems more challenging. Consequently, the

the $L^{2}$-constraint minimization problems become more difficult. For this reason, the literature concerning this topic is far less broad. In \cite{IM}, Ikoma and Miyamoto performed a careful interaction estimate to exclude dichotomy. Then, they modified any minimizing sequence to be an approximated positive solution of the Euler-Lagrange equation and proved
the precompactness of the modified minimizing sequence. After that, Ikoma and Miyamoto \cite{IM1} applied their argument to a minimization problem with two constraint conditions and potentials. In our scope, we would like to mention \cite{AJ1}, where Alves and Ji investigated the existence of minimizer of the $L^{2}$-constraint minimization problem
$$
\Upsilon_a=\inf \left\{J(u): u\in H^{1}(\mathbb{R}^N), \int_{\mathbb{R}^{N}}|u|^{2}dx=a^{2} \right\}
$$
with various types of potentials. Here, $J$ is defined as
\begin{equation*}
J(u)=\frac{1}{2}\int_{\mathbb{R}^N}|\nabla u|^2 \,dx+\frac{1}{2}\int_{\mathbb{R}^N}V(x)|u|^2 \,dx-\frac{1}{q}\int_{\mathbb{R}^N} |u|^q \,dx,
\end{equation*}
where $a>0$,  $q \in (2,2+\frac{4}{N})$ with $N \geq 2$ and the potential $V:\mathbb{R}^N \to [0,+\infty)$ is  a bounded and continuous function.
Moreover, they also established similar results for nonlinear Schr\"{o}dinger equations with a magnetic field and different types of potentials. Inspired by this work, Alves and Thin \cite{AT} studied the existence of multiple normalized solutions to the following class of elliptic problems
$$
\left\{\begin{array}{l}
-\Delta u+V(\epsilon x) u=\lambda u+f(u), \quad \text { in } \mathbb{R}^N, \\
\int_{\mathbb{R}^N}|u|^2 d x=a^2,
\end{array}\right.
$$
where $\epsilon>0$, $V: \mathbb{R}^N \rightarrow[0, \infty)$ is a continuous function, and $f$ is a differentiable function with $L^2$-subcritical growth.  The number of normalized solutions is found to be related to the topology of the set where the potential $V$ attains its minimum value.

Motivated by the works of \cite{Ad, AJ1, AT}, this paper aims to investigate the multiplicity of  normalized solutions to  the following logarithmic Schr\"{o}dinger equation
\begin{align}\label{11*}
 \left\{
\begin{aligned}
&-\epsilon^2\Delta u+V( x)u=\lambda u+u \log u^2, \,\,
\,\,
\hbox{in }\mathbb{R}^N,\\
&\int_{\mathbb{R}^{N}}|u|^{2}dx=a^{2}\epsilon^N,
\end{aligned}
\right.
\end{align}
where $a, \epsilon>0, \lambda \in \mathbb{R}$ is an unknown parameter that appears as a Lagrange multiplier and $V: \mathbb{R}^N \rightarrow[-1, \infty)$ is a continuous function.\\

By introducing a simple change of variable, the problem \eqref{11*} is equivalent to
\begin{align}\label{11}
	\left\{
	\begin{aligned}
		&-\Delta u+V(\epsilon x)u=\lambda u+u \log u^2, \quad
		\quad
		\hbox{in }\mathbb{R}^N,\\
		&\int_{\mathbb{R}^{N}}|u|^{2}dx=a^{2}.
	\end{aligned}
	\right.
\end{align}
Related to the potential function $V$, we make the assumption that $V \in C\left(\mathbb{R}^N, \mathbb{R}\right) \cap L^{\infty}\left(\mathbb{R}^N\right)$ and
satisfies the condition
$$
-1\leq V_0=\inf _{x \in \mathbb{R}^N} V(x)<\liminf _{|x| \rightarrow+\infty} V(x)=V_{\infty}. \leqno{(V_1)}.
$$
Without of loss of generality, we assume that $V(0)=V_0$.

In \cite{Ad}, Alves and da Silva have showed that the function $F_1$ in \eqref{decomposition} is a $\mathrm{N}$-function that satisfies the so-called $\left(\Delta_2\right)$-condition. This fact allows us to consider the reflexive and separable Orlicz space defined as
$$
L^{F_1}\left(\mathbb{R}^N\right)=\left\{u \in L_{l o c}^1\left(\mathbb{R}^N\right):  \int_{\mathbb{R}^N} F_1(|u|) d x<+\infty\right\}.
$$
Hereafter, we define the following work space
$$
X=H^{1}(\mathbb{R}^N) \cap L^{F_1}\left(\mathbb{R}^N\right).
$$
A solution $u$ to problem \eqref{11} with $\int_{\mathbb{R}^N}|u|^2 d x=a^2$ corresponds to a critical point of the functional
\begin{equation*}\label{functional}
J_\epsilon(u)=\frac{1}{2} \int_{\mathbb{R}^N}\left(|\nabla u|^2+(V(\epsilon x)+1)|u|^2\right) d x+\int_{\mathbb{R}^N} F_1(u)d x-\int_{\mathbb{R}^N} F_2(u)d x,\,\, u \in X,
\end{equation*}
restricted to the sphere
$$
S(a)=\left\{u \in X:|u|_2=a\right\}.
$$
It can be easily shown that $J_{\epsilon}\in C^{1}(X, \mathbb{R})$.
Furthermore, by using the inequality
$$
\left|t \log t^2\right| \leq C\left(1+|t|^p\right), \quad p \in\left(2,2^*\right)
$$
and  employing the standard arguments of regularity theory, a critical point of $J_{\epsilon}$ in $X$ is a classical solution of \eqref{11} (see, e.g.,  \cite[Section~1]{AlvesdeMorais}).

The main objective of this paper is to establish the existence of multiple normalized solutions for \eqref{11} under the assumption of $(V_1)$,  using the Lusternik-Schnirelmann category of the sets $M$ and $M_\delta$ defined as
$$
M=\left\{x \in \mathbb{R}^N: V(x)=V_0\right\}
$$
and
$$
M_\delta=\left\{x \in \mathbb{R}^N: \operatorname{dist}(x, M) \leq \delta\right\} .
$$

Now, we can state our main result as follows.

\begin{theorem}\label{T1} Suppose that $V$ satisfies $\left(V_1\right)$. Then $\delta>0$, there exist $a^*>0$ and $\epsilon_0>0$ such that \eqref{11} admits at least $\text{cat}_{M_{\delta}}(M)$ couples $\left(u_j, \lambda_j\right) \in X \times \mathbb{R}$ of weak solutions for $a>a^*$ and $0<\epsilon<\epsilon_0$ with $\int_{\mathbb{R}^N}\left|u_j\right|^2 d x=a^2, \lambda_j<0$ and $J_\epsilon\left(u_j\right)<0$. Moreover, if $u_\epsilon$ denotes one of these solutions and $\xi_\epsilon$ is the global maximum of $\left|u_\epsilon\right|$, then
$$
\lim _{\epsilon \rightarrow 0} V\left(\epsilon \xi_\epsilon\right)=V_0 \text {. }
$$
\end{theorem}

The paper is organized as follows: Section 2 provides a concise overview of Orlicz spaces and introduces the variational framework for our problem.  In Section 3, we investigate a crucial limiting problem that plays a significant role in our approach. Finally, Section 4 presents the proof of our results on the multiplicity of solutions, incorporating the Lusternik-Schnirelmann category.

\vspace{0.5 cm}

\noindent \textbf{Notation:} Throughout this paper, unless otherwise specified, we use the following notations:
\begin{itemize}
	\item $B_r(u)$ is an open ball centered at $u$ with radius $r>0$, $B_r=B_r(0)$.

	\item   $C,C_1,C_2,...$ denote any positive constant, whose value is not relevant.
	
	\item  $|\,\,\,|_p$ denotes the usual norm of the Lebesgue space $L^{p}(\mathbb{R}^N)$, for $p \in [1,+\infty]$,
    $\Vert\,\,\,\Vert$ denotes the usual norm of the Sobolev space $H^{1}(\mathbb{R}^N)$.

	\item $o_{n}(1)$ denotes a real sequence with $o_{n}(1)\to 0$ as $n \to +\infty$.

\end{itemize}

\section{The variational framework and some preliminaries}
In this section, we shall present some fundamental concepts and properties concerning Orlicz spaces. For the more details, please refer to \cite{Ada, FIN, Rao}.
\begin{definition} \label{N-function}
An $N$-function is a continuous function $\Phi: \mathbb{R} \rightarrow[0,+\infty)$ that satisfies the following conditions:\\
\noindent (i) $\Phi$ is convex.

\noindent (ii) $\Phi(t)=0 \Leftrightarrow t=0$.

\noindent (iii) $\displaystyle \lim _{t \rightarrow 0} \frac{\Phi(t)}{t}=0$ and $\displaystyle \lim _{t \rightarrow \infty} \frac{\Phi(t)}{t}=+\infty$.

\noindent (iv) $\Phi$ is an even function.
\end{definition}
We say that an N-function $\Phi$ satisfies the $\Delta_2$-condition, denoted by $\Phi \in\left(\Delta_2\right)$, if
$$
\Phi(2 t) \leq k \Phi(t), \quad \forall t \geq t_0,
$$
for some constants $k>0$ and $t_0 \geq 0$.

The conjugate function $\tilde{\Phi}$ associated with $\Phi$ is obtained through the Legendre's transformation, defined as
$$
\tilde{\Phi}=\max _{t \geq 0}\{s t-\Phi(t)\},\,\, \text { for } s \geq 0 .
$$
It can be shown that that $\tilde{\Phi}$ is also an N-function. The functions $\Phi$ and $\tilde{\Phi}$ are mutually complementary, that is, $\tilde{\tilde{\Phi}}=\Phi$.

For an open set $\Omega \subset \mathbb{R}^N$, we define the Orlicz space associated with the N-function $\Phi$ as follows
$$
L^{\Phi}(\Omega)=\left\{u \in L_{loc}^1(\Omega): \int_{\Omega} \Phi\left(\frac{|u|}{\lambda}\right)dx<+\infty, \,\,\text {for some}\,\, \lambda>0\right\} .
$$
The space $L^{\Phi}(\Omega)$ is a Banach space endowed with the Luxemburg norm, given by
$$
\|u\|_{\Phi}=\inf \left\{\lambda>0: \int_{\Omega} \Phi\left(\frac{|u|}{\lambda}\right)dx\leq 1\right\}.
$$
In the Orlicz spaces, we also have H\"{o}lder and Young type inequalities, namely
$$
s t \leq \Phi(t)+\tilde{\Phi}(s), \,\,\text{for}\,\,\forall s, t \geq 0
$$
and
$$
\left|\int_{\Omega} u vdx\right| \leq 2\|u\|_{\Phi}\|v\|_{\tilde{\Phi}}, \,\,\text{for}\,\, \forall u \in L^{\Phi}(\Omega) \,\, \text {and} \,\,\forall v \in L^{\bar{\Phi}}(\Omega).
$$
When $\Phi$, $\tilde{\Phi} \in\left(\Delta_2\right)$, the space $L^{\Phi}(\Omega)$ is reflexive and separable. Moreover, the $\Delta_2$-condition implies that
$$
L^{\Phi}(\Omega)=\left\{u \in L_{loc}^1(\Omega): \int_{\Omega} \Phi(|u|)dx<+\infty\right\}
$$
and
$$
u_n \rightarrow u \,\,\text{in}\,\, L^{\Phi}(\Omega) \Leftrightarrow \int_{\Omega} \Phi\left(\left|u_n-u\right|\right)dx\rightarrow 0 .
$$
We would like to mention an important relation involving $\mathrm{N}$-functions, which will be used later. Let $\Phi$ be a $\mathrm{N}$-function of $C^1$ class and $\tilde{\Phi}$ is its conjugate function. Suppose that
\begin{equation}\label{ine01}
1<l \leq \frac{\Phi^{\prime}(t) t}{\Phi(t)} \leq m<N, \quad t \neq 0,
\end{equation}
then $\Phi, \tilde{\Phi} \in\left(\Delta_2\right)$.

Finally, setting the functions
$$
\xi_0(t)=\min \left\{t^l, t^m\right\} \,\,\text {and} \,\,\xi_1(t)= \max \left\{t^l, t^m\right\}, \quad t \geq 0,
$$
it is well known that under the condition \eqref{ine01}, the function $\Phi$ satisfies the inequality below
\begin{equation}\label{ine02}
\xi_0\left(\|u\|_{\Phi}\right) \leq \int_{\mathbb{R}^N} \Phi(u) \leq \xi_1\left(\|u\|_{\Phi}\right), \quad \forall u \in L^{\Phi}(\Omega) .
\end{equation}

Fixed $\delta>0$ sufficiently small, let us define the functions  $F_1$ and $F_2$ as follows

$$
F_1(s)=
\left\{ \begin{array}{lc}
0, & \; s\leq 0, \\
-\frac{1}{2}s^2\log s^2, & \; 0< \vert s\vert<\delta, \\
-\frac{1}{2}s^2(
\log\delta^2+3)+2\delta \vert s\vert-\frac{1}{2}\delta^2,  & \;  s\geq \delta,
\end{array} \right.
$$

\noindent and
$$
F_2(s)=
\left\{ \begin{array}{lc}
0, & \;  \vert s\vert\leq\delta, \\
\frac{1}{2}s^2\log(s^2/\delta^2)+2\delta s-\frac{3}{2}s^2-\frac{1}{2}\delta^2,  & \;  \vert s\vert\geq \delta,
\end{array} \right.
$$
then
$$
F_2(s)-F_1(s)=\frac{1}{2} s^2 \log s^2, \quad \forall s \in \mathbb{R} .
$$
Moreover, $F_1$ and $F_2$ satisfy the following properties:
\begin{itemize}
\item[\rm ($P_{1}$)] $F_1$ is an even function with $F_1^{\prime}(s) s \geq 0$ and $F_1 \geq 0$. Moreover, $F_1 \in C^1(\mathbb{R}, \mathbb{R})$ is  convex if $\delta \approx 0^{+}$.

\item[\rm ($P_{2}$)] $F_2 \in C^1(\mathbb{R}, \mathbb{R}) \cap C^2((\delta,+\infty), \mathbb{R})$ and for each $p \in\left(2,2^*\right)$, there exists $C_p>0$ such that
\begin{equation}\label{subcritical}
\left|F_2^{\prime}(s)\right| \leq C_p|s|^{p-1}, \quad \forall s \in \mathbb{R} .
\end{equation}

\item[\rm ($P_{3}$)] $s \mapsto \frac{F_2^{\prime}(s)}{s}$ is a nondecreasing function for $s>0$ and a strictly increasing function for $s>\delta$.

\item[\rm ($P_{4}$)] $\displaystyle \lim _{s \rightarrow \infty} \frac{\dot{F}_2^{\prime}(s)}{s}=\infty$.
\end{itemize}

An important result involving the function $F_1$ is as follows, which can be found in \cite{Ad}. However, for the convenience of the reader, we will provide its proof here.

\begin{lemma}\label{N-function}
The function $F_1$ is a N-function. Furthermore, it holds that $F_1, \tilde{F}_1 \in\left(\Delta_2\right)$.
\end{lemma}
\begin{proof}
	By direct computation, it can be showed that $F_1$ verifies $i)-iv)$ of the Definition \ref{N-function}. To complete the proof,  we need to show that $F_1$ satisfies the relation (\ref{ine01}), which  implies that $F_1, \tilde{F}_1 \in\left(\Delta_2\right)$. First of all, observe that
	$$
	F'_1(s):=\left\{\begin{aligned}
		&-(\log s^2+1)s,\quad &0<s<\delta,\\
		&-s(\log \delta^2 +3) + 2\delta & s\geq \delta.
	\end{aligned}
	\right.
	$$
	Next, we will analyze the cases $0<s<\delta$ and $s\geq \delta$ separately.
	\\
	\textbf{Case 1}: $0<s<\delta$. \\
  In this case,
	\begin{equation*}
		\frac{F'_1(s)s}{F_1(s)}=2+\frac{1}{\log s},
	\end{equation*}
	which implies the existence of $l_1>1$ satisfying
	\begin{equation}\label{1}
		1<l_1\leq\frac{F'_1(s)s}{F_1(s)}\leq m_1:=\sup_{0<s<\delta}\left(2+\frac{1}{\log s}\right)\leq 2,
	\end{equation}
	for $\delta>0$ small enough.
	\\
	\textbf{Case 2}: $s\geq\delta$. \\
In this case,
	\begin{equation*}
		\frac{F'_1(s)s}{F_1(s)}=\frac{-(\log \delta^2 +3)s^2 + 2\delta s}{-\frac{1}{2}(\log \delta^2 +3)s^2 + 2\delta s-\frac{1}{2}\delta^2}.
	\end{equation*}
	From this, we have
	$$\sup_{s\geq \delta}\frac{F'_1(s)s}{F_1(s)}\leq  \sup_{s\geq \delta} \left(\frac{-(\log \delta^2 +3)s^2 + 2\delta s+(2\delta s - \delta^2)}{-\frac{1}{2}(\log \delta^2 +3)s^2 + 2\delta s-\frac{1}{2}\delta^2}\right)\leq 2.$$
	Since
	$$
	\lim_{s \to +\infty}\frac{F'_1(s)s}{F_1(s)}=2 \quad \mbox{and} \quad \frac{F'_1(s)s}{F_1(s)}> 1, \quad \forall s>0,
	$$
	we obtain
	$$
	1<\inf_{s>0}\frac{F'_1(s)s}{F_1(s)}.
	$$
	
	The last inequalities ensure the existence of $l \in (1,2)$ such that
	\begin{equation} \label{2}
		1<l\leq\frac{F'_1(s)s}{F_1(s)}\leq 2,\,\,\,\, \forall s>0.
	\end{equation}
	Since $F_1$ is an even function, the inequalities above hold  for any $s\neq 0$, and thus, the proof is complete.
\end{proof}

As a byproduct of the previous lemma, we have the following corollary:
\begin{corollary} \label{CorF1}   The functional $\Theta:L^{F_1}\left(\mathbb{R}^N\right) \to \mathbb{R}$ given by
$$
\Theta(u)=\int_{\mathbb{R}^N}F_1(u)\,dx
$$
is $C^{1}(L^{F_1}\left(\mathbb{R}^N\right),\mathbb{R})$ with
$$
\Theta'(u)v=\int_{\mathbb{R}^N}F_1'(u)v\,dx, \quad \forall u\,v \in L^{F_1}\left(\mathbb{R}^N\right),
$$
where $L^{F_1}\left(\mathbb{R}^N\right)$ denotes the Orlicz space associated with $F_1$ endowed with the  Luxemburg norm $\|\cdot\|_{F_1}$.
\end{corollary}

In the sequel, in order to avoid the points $u \in H^1\left(\mathbb{R}^N\right)$ that satisfy $F_1(u) \notin L^1\left(\mathbb{R}^N\right)$, we will restrict the functional $J_{\epsilon}$ given in section 2 to the space $X=H^1\left(\mathbb{R}^N\right) \cap L^{F_1}\left(\mathbb{R}^N\right)$, which will be denoted by $I_{\epsilon}$, that is, $\left.I_{\epsilon} \equiv J_{\epsilon}\right|_{X}$.  Hereafter, let us consider on $X$ the norm
$$
\|\cdot\|:=\|\cdot\|_{H^1}+\|\cdot\|_{F_1},
$$
where $\|\cdot\|_{H^1}$ denotes the usual norm in $H^1\left(\mathbb{R}^N\right)$.

From Lemma \ref{N-function}, it follows that $\left(X,\|\cdot\|\right)$ is a reflexive and separable Banach space. Additionally, note that the embeddings $X \hookrightarrow H^1\left(\mathbb{R}^N\right)$ and $X \hookrightarrow L^{F_1}\left(\mathbb{R}^N\right)$ are continuous.

Consequently, from the conditions on $F_1$ and $V$, it is easy to check that $I_{\epsilon} \in C^1\left(X, \mathbb{R}\right)$ with
$$
I_{\epsilon}^{\prime}(u) v=\int_{\mathbb{R}^N}\Big(\nabla u \nabla v+(V(\epsilon x)+1) u v\Big) dx+\int_{\mathbb{R}^N} F_1^{\prime}(u) vdx-\int_{\mathbb{R}^N} F_2^{\prime}(u) vdx, \quad \forall v \in X .
$$

\section{The limit problem}

In this section, we first  investigate the existence of normalized solution to the problem
\begin{align}\label{12}
 \left\{
\begin{aligned}
&-\Delta u+\mu u=\lambda u+u \log u^2,
\quad
\hbox{in}\,\,\mathbb{R}^N,\\
&\int_{\mathbb{R}^{N}}|u|^{2}dx=a^{2},
\end{aligned}
\right.
\end{align}
where $N \geq 1, a>0, \mu \geq -1$ and $\lambda \in \mathbb{R}$ is an unknown parameter that appears as a Lagrange multiplier.

It is well known that a solution $u$ to problem \eqref{12} corresponds to a critical point of the following $C^1$ functional
$$
I_\mu(u)=\frac{1}{2} \int_{\mathbb{R}^N}\left(|\nabla u|^2+(\mu+1)u^2\right) d x+\int_{\mathbb{R}^N} F_1(u)d x-\int_{\mathbb{R}^N} F_2(u) d x,\,\, u \in X
$$
restricted to the sphere $S(a)$ given by
$$
S(a)=\left\{u \in X:|u|_2=a\right\} .
$$
Our main result in this section is as follows.

\begin{theorem}\label{T2} There exists $a^*(\mu)>0$ such that problem \eqref{12} has a couple $(u, \lambda)$ solution when $a\geq a^*$, where $u$ is positive, radial and $\lambda<0$.
\end{theorem}

The proof of the above theorem will be divided into several lemmas.
\begin{lemma}\label{coercive}
The functional $I_\mu$ is coercive and bounded from below on $S(a)$.
\end{lemma}
\begin{proof} First of all, we recalling that by $(P_2)$, for each fixed $p \in (2, 2+\frac{4}{N})$, there exists a constant $C_{p}>0$ such that
\begin{equation*}\label{eq5}
|F'_2(s)|\leq C_{p}|s|^{p-1}, \quad \forall s \in \mathbb{R}.
\end{equation*}
Moreover, by the Gagliardo-Nirenberg inequality, one has
$$
|u|_l^l \leq C|u|_2^{\left(1-\beta_l\right) l}|\nabla u|_2^{\beta_l l}, \quad  \forall u \in H^{1}(\mathbb{R}^N),
$$
for some positive constant $C=C(N, l)>0$, where $\beta_l=N\left(\frac{1}{2}-\frac{1}{l}\right)$, $l \in\left[2, \frac{2 N}{N-2}\right)$ if $N \geq 3$ and $l \geq 2$  if $N =1, 2$. Hence,
\begin{align*}
I_\mu(u)=&\frac{1}{2} \int_{\mathbb{R}^N}\left(|\nabla u|^2+(\mu+1)u^2\right) d x+\int_{\mathbb{R}^N} F_1(u)d x-\int_{\mathbb{R}^N} F_2(u) d x\\
\geq & \frac{1}{2} \int_{\mathbb{R}^N}|\nabla u|^2 d x+\int_{\mathbb{R}^N} F_1(u)d x-C C_1 a^{\left(1-\beta_p\right) p}\left(\int_{\mathbb{R}^N}|\nabla u|^2 d x\right)^{\frac{p\beta_p }{2}}.
\end{align*}
Since $p \in\left(2,2+\frac{4}{N}\right)$, it is clear that $p\beta_p <2$. Moreover, by \eqref{ine02}, we have that $\int_{\mathbb{R}^N} F_1(u)d x \rightarrow +\infty$ as $\Vert u\Vert_{F_{1}}\rightarrow\infty$. These facts ensure the coercivity and boundedness of $I_\mu$ from below on the sphere $S(a)$.
\end{proof}

The above lemma  ensures the well-definedness of the real number
$$
\mathcal{I}_{\mu, a}=\inf _{u \in S(a)} I_\mu(u).
$$
 Next, we will establish some properties of $\mathcal{I}_{\mu, a}$ with respect to the parameter $\mu \geq -1$.
\begin{lemma}\label{negative12}
For each $\mu^* > -1$, there exists $a^*=a^*(\mu^*)>0$ such that $\mathcal{I}_{\mu,a}<0$ for all $a \geq a^*$ and $\mu \in [-1,\mu^*]$.
\end{lemma}
\begin{proof}
Fixing $\varphi \in C_{0}^{\infty}(\mathbb{R}^N) \setminus \{0\}$ and $t>0$, we have
$$
I_\mu(t \varphi)\leq \frac{t^2}{2}\int_{\mathbb{R}^N}\left(|\nabla \varphi|^2+(\mu^*+1)\varphi^2\right) d x-\frac{t^2}{2}\int_{\mathbb{R}^N}\varphi^2 \log \varphi^2 dx-t^2 \log t \int_{\mathbb{R}^N}\varphi^2dx \to -\infty
\,\, \mbox{as} \,\, t \to +\infty.
$$
Thus, there exists $t^*=t^*(\mu)>0$ such that
$$
I_\mu(t \varphi) < -1, \quad \forall t \geq t^*.
$$
Now, let $a^*=t^*|\varphi|_2$, if $a\geq a^*$, from the above arguments, we derive that  $\mathcal{I}_{\mu, a}<0$ for all $\mu \in [-1,\mu^*]$.

\end{proof}

\begin{lemma}\label{subadditive}
Fix $\mu \in[-1, \mu^*]$ and let $0<a_* \leq a_1<a_2$. Then, $\frac{a_1^2}{a_2^2} \mathcal{I}_{\mu, a_2}<\mathcal{I}_{\mu, a_1}<0$.
\end{lemma}
\begin{proof} Let $\xi>1$ such that $a_2=\xi a_1$ and $\left(u_n\right) \subset S\left(a_1\right)$ be a nonnegative minimizing sequence with respect to $\mathcal{I}_{\mu, a_1}$, which exists because $I_\mu(u)=I_\mu(|u|)$ for all $u \in X$, that is,
$$
I_\mu\left(u_n\right) \rightarrow \mathcal{I}_{\mu, a_1}, \text { as } n \rightarrow+\infty .
$$
Setting $v_n=\xi u_n$, obviously $v_n \in S\left(a_2\right)$. By a direct computation, one gets
\begin{equation*}
\mathcal{I}_{\mu, a_2} \leq I_\mu\left(v_n\right)=\xi^2 I_\mu\left(u_n\right)-\frac{1}{2}\xi^2 \log \xi^2\int_{\mathbb{R}^N}\vert u_n\vert^2 \,dx=\xi^2 I_\mu\left(u_n\right)-\frac{1}{2}a_1^{2}\xi^2 \log \xi^2.
\end{equation*}
Letting $n \rightarrow+\infty$, it follows from $\xi>1$ that
$$
\mathcal{I}_{\mu, a_2} \leq \xi^2 \mathcal{I}_{\mu, a_1}-\frac{1}{2}a_1^{2}\xi^2 \log \xi^2<\xi^2 \mathcal{I}_{\mu, a_1},
$$
that is,
$$
\frac{a_1^2}{a_2^2} \mathcal{I}_{\mu, a_2}<\mathcal{I}_{\mu, a_1},
$$
which proves the lemma.
\end{proof}
Similar to the proof of  \cite[Lemma~3.1]{Shu}, we have the following lemma.
\begin{lemma}\label{Brezis-Lieb}
Let $(u_n)$ be a bounded sequence in $X$ such that $u_n \rightarrow u$ a.e. in $\mathbb{R}^N$ and $\{u_n^2 \log u_n^2\}$ is a bounded sequence in $L^1\left(\mathbb{R}^N\right)$. Then,
$$
\lim _{n \rightarrow \infty} \int_{\mathbb{R}^N}\left[u_n^2 \log u_n^2-\left|u_n-u\right|^2 \log \left|u_n-u\right|^2\right] \mathrm{d} x=\int_{\mathbb{R}^N} u^2 \log u^2 \mathrm{~d} x.
$$
\end{lemma}

Our next result is a compactness theorem on $S(a)$, which is a useful tool for studying both the autonomous and non-autonomous cases.

\begin{theorem}\label{compactness} (Compactness theorem on $S(a)$ ) Let $a>a^{*}$ and $\left(u_n\right) \subset S(a)$ be a minimizing sequence with respect to $\mathcal{I}_{\mu, a}$. Then, for some subsequence either\\
\noindent (i) $\left(u_n\right)$ is strongly convergent in $X$,\\
or\\
\noindent (ii) There exists $\left(y_n\right) \subset \mathbb{R}^N$ with $\left|y_n\right| \rightarrow+\infty$ such that the sequence $v_n(x)=u_n\left(x+y_n\right)$ strongly converges to a function $v \in S(a)$ in $X$ with $I_\mu(v)=\mathcal{I}_{\mu, a}$.
\end{theorem}
\begin{proof}
Since $\mathcal{I}_\mu$ is coercive on $S(a)$, the sequence $\left(u_n\right)$ is bounded in $X$, and so, $u_n \rightharpoonup u$ in $X$ for some subsequence. If $u \neq 0$ and $|u|_2=b \neq a$, we must have $b \in(0, a)$, and by the Br\'{e}zis-Lieb Lemma (see \cite{Willem}),
$$
\left|u_n\right|_2^2=\left|u_n-u\right|_2^2+|u|_2^2+o_n(1).
$$
Furthermore, by Lemma \ref{Brezis-Lieb}, one has

$$
\int_{\mathbb{R}^N}u_n^2 \log u_n^2 dx=\int_{\mathbb{R}^N}\vert u_n-u\vert^2 \log \vert u_n-u\vert^2 dx+\int_{\mathbb{R}^N} u^2 \log u^2 d x+o_n(1).
$$
Setting $v_n=u_n-u, d_n=\left|v_n\right|_2$ and supposing that $d_n \rightarrow d$, we infer that $a^2=b^2+d^2$ and $d_n \in(0, a)$ for $n\in \mathbb{N}$ large enough. Hence,
$$
\mathcal{I}_{\mu, a}+o_n(1)=I_\mu\left(u_n\right)=I_\mu\left(v_n\right)+I_\mu(u)+o_n(1) \geq \mathcal{I}_{\mu, d_n}+\mathcal{I}_{\mu, b}+o_n(1),
$$
and by Lemma \ref{subadditive},
$$
\mathcal{I}_{\mu, a}+o_n(1) \geq \frac{d_n^2}{a^2} \mathcal{I}_{\mu, a}+\mathcal{I}_{\mu, b}+o_n(1) .
$$
Letting $n \rightarrow+\infty$, one finds
\begin{equation}\label{inequality21}
\mathcal{I}_{\mu, a} \geq \frac{d^2}{a^2} \mathcal{I}_{\mu, a}+\mathcal{I}_{\mu, b} .
\end{equation}
Since $b \in(0, a)$, employing again Lemma \ref{subadditive} in \eqref{inequality21}, we arrive at the following inequality
$$
\mathcal{I}_{\mu, a}>\frac{d^2}{a^2} \mathcal{I}_{\mu, a}+\frac{b^2}{a^2} \mathcal{I}_{\mu, a}=\left(\frac{d^2}{a^2}+\frac{b^2}{a^2}\right) \mathcal{I}_{\mu, a}=\mathcal{I}_{\mu, a},
$$
which is absurd. This asserts that $|u|_2=a$, or equivalently, $u \in S(a)$.
As $\left|u_n\right|_2=|u|_2=a, u_n \rightharpoonup u$ in $L^2\left(\mathbb{R}^N\right)$ and $L^2\left(\mathbb{R}^N\right)$ is reflexive, we must have that
\begin{equation}\label{inequality22}
u_n \rightarrow u \quad \text { in } L^2\left(\mathbb{R}^N\right).
\end{equation}
This combined with interpolation theorem in the Lebesgue spaces and \eqref{subcritical} leads to
\begin{equation}\label{inequality23}
\int_{\mathbb{R}^N} F_{2}\left(u_n\right) d x \rightarrow \int_{\mathbb{R}^N} F_{2}(u) d x.
\end{equation}
These limits together with the limit $\mathcal{I}_{\mu, a}=\displaystyle \lim _{n \rightarrow+\infty} I_\mu\left(u_n\right)$ and the fact that $F_1 \geq 0$ provide
$$
\mathcal{I}_{\mu, a} \geq I_\mu(u) \text {. }
$$
As $u \in S(a)$, we infer that $I_\mu(u)=\mathcal{I}_{\mu, a}$, then
$$
\lim _{n \rightarrow+\infty} I_\mu\left(u_n\right)=I_\mu(u),
$$
and so,
$$
\int_{\mathbb{R}^N}|\nabla u_n|^2\,dx \to \int_{\mathbb{R}^N}|\nabla u|^2\,dx
$$
and
$$
\int_{\mathbb{R}^N}F_1(u_n)\,dx \to \int_{\mathbb{R}^N}F_1(u)\,dx.
$$
Since $F_1\in (\Delta_2)$, \eqref{inequality22} combines with the last two limits permit to  conclude that $u_n \rightarrow u$ in $X$.\\

Now, let us assume that $u=0$, that is, $u_n \rightharpoonup 0$ in $X$. \\

\begin{claim} \label{Claim 3.3.} There exists a positive constant $C>0$ such that
\begin{equation}\label{inequality24}
\int_{\mathbb{R}^N} F_{2}\left(u_n\right) d x \geq C, \quad \text { for } n \in \mathbb{N} \text { large. }
\end{equation}
\end{claim}
Otherwise, if \eqref{inequality24} is not true, then there exists a subsequence of $\left(u_n\right)$, still denoted by itself, such that
$$
\int_{\mathbb{R}^N} F_2\left(u_n\right) d x \rightarrow 0, \text { as } n \rightarrow+\infty.
$$
Now, recalling that
$$
0>\mathcal{I}_{\mu, a}=I_\mu\left(u_n\right) +o_n(1)\geq-\int_{\mathbb{R}^N} F_2\left(u_n\right) d x +o_n(1), \quad n \in \mathbb{N}\,\,\text{large},
$$
we arrive at a contradiction, and Claim \ref{Claim 3.3.} is proved. Thanks to Claim \ref{Claim 3.3.},  there are $R, \beta>0$ and $y_n \in \mathbb{R}^N$ such that
\begin{equation}\label{inequality25}
\int_{B_R\left(y_n\right)}\left|u_n\right|^2 d x \geq \beta, \text { for all } n \in \mathbb{N},
\end{equation}
otherwise we would  have $u_n \rightarrow 0$ in $L^t\left(\mathbb{R}^N\right)$ for all $t \in\left(2,2^*\right)$ that implies $F_{2}\left(u_n\right) \rightarrow 0$ in $L^1\left(\mathbb{R}^N\right)$, which contradicts \eqref{inequality24}. Since $u=0$, the inequality \eqref{inequality25} together with the Sobolev embedding implies that $\left(y_n\right)$ is unbounded. From this, considering $\tilde{u}_n(x)=u\left(x+y_n\right)$, clearly $\left(\tilde{u}_n\right) \subset S(a)$ and it is also a minimizing sequence with respect to $\mathcal{I}_{\mu, a}$. Moreover, there exists $v \in X \backslash\{0\}$ such that
$$
\tilde{u}_n \rightharpoonup v \text { in } X \text { and } \tilde{u}_n(x) \rightarrow v(x) \text { a.e. in } \mathbb{R}^N \text {. }
$$
Following as in the first part of the proof, we derive that $\tilde{u}_n \rightarrow v$ in $X$ which concludes the proof.

\end{proof}

\subsection{Proof of Theorem \ref{T2}}
 By Lemma \ref{coercive}, there exists a bounded minimizing sequence $\left(u_n\right) \subset S(a)$ with respect to $\mathcal{I}_{\mu, a}$, such that $I_\mu\left(u_n\right) \rightarrow \mathcal{I}_{\mu, a}$. Now, applying Theorem \ref{compactness}, there exists $u \in S(a)$ with $I_\mu(u)=\mathcal{I}_{\mu, a}$. Therefore, by the Lagrange multiplier, there exists $\lambda_a \in \mathbb{R}$ such that
\begin{equation}\label{inequality26}
I_\mu^{\prime}(u)=\lambda_a \Psi^{\prime}(u), \,\,\text { in } \,\, X^{\prime}
\end{equation}
where $\Psi:  X \rightarrow \mathbb{R}$ is defined as
$$
\Psi(u)=\frac{1}{2}\int_{\mathbb{R}^N}|u|^2 d x, \quad u \in  X.
$$
Thus, according to \eqref{inequality26}, we have
$$
-\Delta u+\mu u=\lambda_a u+u \log u^2, \quad \text { in } \mathbb{R}^N \text {. }
$$
Moreover, since $I_\mu(u)=\mathcal{I}_{\mu, a}<0$, it follows that $\lambda_a<0$. In fact, the equality above implies
$$
\int_{\mathbb{R}^N}(|\nabla u|^2+\mu|u|^2)\,dx=\lambda_a a^2+\int_{\mathbb{R}^N} u^2 \log u^2\,dx.
$$
Using the definition of $F_1$ and $F_2$, we know that
$$
F_2(s)-F_1(s)=\frac{1}{2}s^2 \log s^2, \quad \forall s \in \mathbb{R},
$$
then,
$$
\frac{1}{2}\int_{\mathbb{R}^N}(|\nabla u|^2+\mu|u|^2)\,dx+\int_{\mathbb{R}^N}F_1(u)\,dx-\int_{\mathbb{R}^N}F_2(u)\,dx=\frac{\lambda_a a^2}{2}.
$$
This implies
$$
\mathcal{I}_{\mu, a}=I_\mu(u) \geq \frac{\lambda_a a^2}{2},
$$
which shows that $\lambda_a<0$.

Now, we aim to prove that $u$ can be chosen as a positive function. From the definition of the functional $I_\mu$, it is easy to check that $I_\mu(|u|)=I_\mu(u)$. Moreover, since $u \in S(a)$ implies that $\vert u\vert \in S(a)$, we have
$$
\mathcal{I}_{\mu, a}=I_\mu(u)=I_\mu(|u|) \geq \mathcal{I}_{\mu, a},
$$
which shows that $I_\mu(|u|)=\mathcal{I}_{\mu, a}$, and so, we can replace $u$ by $|u|$. Furthermore, if $u^*$ denotes the Schwarz's symmetrization of $u$, we have
$$
\int_{\mathbb{R}^N}|\nabla u|^2 d x \geq \int_{\mathbb{R}^N}\left|\nabla u^*\right|^2 d x, \quad \int_{\mathbb{R}^N}|u|^2 d x=\int_{\mathbb{R}^N}\left|u^*\right|^2 d x,
$$
and
$$
\int_{\mathbb{R}^N} F_{1}(u) d x=\int_{\mathbb{R}^N} F_{1}\left(u^*\right) d x,\quad \int_{\mathbb{R}^N} F_{2}(u) d x=\int_{\mathbb{R}^N} F_{2}\left(u^*\right) d x,
$$
then $u^* \in S(a)$ and $I_\mu\left(u^*\right)=\mathcal{I}_{\mu, a}$, from where it follows that we can replace $u$ by $u^*$. Thus, we may show that $u$ is radial. Moreover, by an analogous reasoning as used in the proof of \cite[Theorem 3.1]{AJ1} and \cite[Section 3.1]{DMS}, using a suitable version of maximum principle \cite[Theorem 1]{V}, we deduce that $u$ is positive in $\mathbb{R}^N$.

An immediate consequence of Theorem \ref{T2} is the following corollary.

\begin{corollary}\label{coro1}
Fix $a>a^{*}>0$ and let $-1\leq \mu_1<\mu_2 \leq \mu^*$. Then, $\mathcal{I}_{\mu_1, a}<\mathcal{I}_{\mu_2, a}<0$.
\end{corollary}
\begin{proof}
 Let $u_{\mu_2, a} \in S(a)$ satisfying $I_{\mu_2}\left(u_{\mu_2, a}\right)=\mathcal{I}_{\mu_2, a}$. Then,
$$
\mathcal{I}_{\mu_1, a} \leq I_{\mu_1}\left(u_{\mu_2, a}\right)<I_{\mu_2}\left(u_{\mu_2, a}\right)=\mathcal{I}_{\mu_2, a}.
$$

\section{The proof of Theorem \ref{T1}}
In this section, we will establish some properties of the functional $I_{\epsilon}: X \rightarrow \mathbb{R}$ defined as
$$
I_{\epsilon}(u)=\frac{1}{2} \int_{\mathbb{R}^N}\Big(|\nabla u|^2+(V(\epsilon x)+1)|u|^2\Big) d x+\int_{\mathbb{R}^N} F_1(u)d x-\int_{\mathbb{R}^N} F_2(u)d x,
$$
constrained to the sphere $S(a)$.\\

In what follows, we will denote by $J_0, J_{\infty}: X \rightarrow \mathbb{R}$ the following functionals
$$
I_0(u)=\frac{1}{2} \int_{\mathbb{R}^N}\Big(|\nabla u|^2+(V_0+1)|u|^2\Big) d x+\int_{\mathbb{R}^N} F_1(u)d x-\int_{\mathbb{R}^N} F_2(u)d x
$$
and
$$
I_{\infty}(u)=\frac{1}{2} \int_{\mathbb{R}^N}\Big(|\nabla u|^2+(V_{\infty}+1)|u|^2\Big) d x+\int_{\mathbb{R}^N} F_1(u)d x-\int_{\mathbb{R}^N} F_2(u)d x.
$$
Moreover, let us denote by $\Upsilon_{\epsilon, a}, \Upsilon_{0, a}$ and $\Upsilon_{\infty, a}$ the following real numbers
$$
\Upsilon_{\epsilon, a}=\inf _{u \in S(a)}I_\epsilon(u), \quad \Upsilon_{0, a}=\inf _{u \in S(a)}I_0(u) \quad \text {and } \quad  \Upsilon_{\infty, a}=\inf _{u \in S(a)}I_{\infty}(u).
$$
Because $-1<V_{\infty}<+\infty$, fixing $\mu^*=\displaystyle \sup_{x \in \mathbb{R}^N} V(x)>-1$ in Lemma \ref{negative12},  Corollary \ref{coro1} asserts the inequality
\begin{equation}\label{comparison}
\Upsilon_{0, a}<\Upsilon_{\infty, a}<0, \,\, \text {if}\,\, a>a^{*}>0.
\end{equation}
\end{proof}
Next, we fix $0<\rho_1=\frac{1}{2}\left(\Upsilon_{\infty, a}-\Upsilon_{0, a}\right)$.

Our first lemma in this section establishes several crucial relationships among the levels $\Upsilon_{\epsilon, a}$, $\Upsilon_{\infty, a}$, and $\Upsilon_{0, a}$.
\begin{lemma}\label{relation}
 For any $a>a^{*}>0$, $\underset{\epsilon\rightarrow 0}{\limsup}\Upsilon_{\epsilon, a} \leq \Upsilon_{0, a}$ and there exists $\epsilon_0>0$ such that $\Upsilon_{\epsilon, a}<\Upsilon_{\infty, a}$  for all $\epsilon \in\left(0, \epsilon_0\right)$.
\end{lemma}
\begin{proof}For any $a>a^{*}>0$, let $u_0 \in S(a)$ with $I_0\left(u_0\right)=\Upsilon_{0, a}$. Then,
$$
\Upsilon_{\epsilon, a} \leq I_\epsilon\left(u_0\right)=\frac{1}{2} \int_{\mathbb{R}^N}\left(\left|\nabla u_0\right|^2+(V(\epsilon x)+1)\left|u_0\right|^2\right) d x+\int_{\mathbb{R}^N} F_1(u_0)d x-\int_{\mathbb{R}^N} F_2(u_0)d x.
$$
Taking the limit as $\epsilon \rightarrow 0^{+}$, we find
$$
\limsup _{\epsilon \rightarrow 0^{+}} \Upsilon_{\epsilon, a} \leq \lim _{\epsilon \rightarrow 0^{+}} I_\epsilon\left(u_0\right)=I_0\left(u_0\right)=\Upsilon_{0,a}.
$$
The inequality $\Upsilon_{\epsilon, a}<\Upsilon_{\infty, a}$ for $\epsilon>0$ small enough follows from \eqref{comparison} together with the last inequality.
\end{proof}
\begin{lemma}\label{nontrivial}
 For any $a>a^{*}>0$, fix $\epsilon \in\left(0, \epsilon_0\right)$ and let $\left(u_n\right) \subset S(a)$ be such that $I_\epsilon\left(u_n\right) \rightarrow c$ with $c<\Upsilon_{0, a}+\rho_1<0$. If $u_n \rightharpoonup u$ in $X$, then $u \neq 0$.
\end{lemma}
\begin{proof}
Assume, by contradiction, that $u=0$. Then,
$$
\Upsilon_{0, a}+\rho_1+o_n(1)>I_\epsilon\left(u_n\right)=I_{\infty}\left(u_n\right)+\frac{1}{2} \int_{\mathbb{R}^N}\left(V(\epsilon x)-V_{\infty}\right)\left|u_n\right|^2 d x .
$$
Using $\left(V_1\right)$, for any given $\zeta>0$, there exists $R>0$ such that
$$
V(x) \geq V_{\infty}-\zeta, \text { for all }|x| \geq R \text {. }
$$
Hence
$$
\begin{aligned}
\Upsilon_{0, a}+\rho_1+o_n(1)>I_\epsilon\left(u_n\right) \geq & I_{\infty}\left(u_n\right)+\frac{1}{2} \int_{B_{R / \epsilon}(0)}\left(V(\epsilon x)-V_{\infty}\right)\left|u_n\right|^2 d x\\
& -\frac{\zeta}{2} \int_{B_{R / \varsigma}^\epsilon(0)}\left|u_n\right|^2 d x.
\end{aligned}
$$
Recalling that $\left(u_n\right)$ is bounded in $X$ and $u_n \rightarrow 0$ in $L^l\left(B_{R / \epsilon}(0)\right)$ for all $l \in\left[1,2^*\right)$, it follows that
$$
\Upsilon_{0, a}+\rho_1+o_n(1) \geq I_{\infty}\left(u_n\right)-\zeta C \geq \Upsilon_{\infty, a}-\zeta C
$$
for some $C>0$. Since $\zeta>0$ is arbitrary, we deduce that
$$
\Upsilon_{0, a}+\rho_1 \geq \Upsilon_{\infty, a},
$$
which contradicts the definition of $\rho_1$. Thus, the weak limit $u$ of $\left(u_n\right)$ is nontrivial.

\end{proof}
\begin{lemma}\label{nontrivial1}
Let $\left(u_n\right) \subset S(a)$ be a $(P S)_c$ sequence for $I_\epsilon$ constrained to $S(a)$ with \linebreak $c<\Upsilon_{0, a}+\rho_1<0$ and $u_n \rightharpoonup u_\epsilon$ in $X$, that is,
$$
I_\epsilon\left(u_n\right) \rightarrow c \text { as } n \rightarrow+\infty \quad \text { and }\left\|\left.I_\epsilon^{\prime}\right|_{S(a)} \left(u_n\right)\right\| \rightarrow 0 \quad \text { as } n \rightarrow+\infty .
$$
If $v_n=u_n-u_\epsilon \not\rightarrow 0$ in $X$, then  there exists $\beta>0$ independent of $\epsilon \in\left(0, \epsilon_0\right)$ such that, by decreasing $\epsilon_0$ if necessary,
$$
\liminf _{n \rightarrow+\infty}\left|u_n-u_\epsilon\right|_2^2 \geq \beta .
$$
\end{lemma}
\begin{proof}
Letting $\Psi: X \rightarrow \mathbb{R}$ be the functional defined as
$$
\Psi(u)=\frac{1}{2} \int_{\mathbb{R}^N}|u|^2 dx,
$$
it follows that $S(a)=\Psi^{-1}\left(\left\{a^2 / 2\right\}\right)$. Then, by Willem \cite[Proposition~5.12]{Willem}, there exists $\left(\lambda_n\right) \subset \mathbb{R}$ such that
$$
\left\|I_\epsilon^{\prime}\left(u_n\right)-\lambda_n \Psi^{\prime}\left(u_n\right)\right\|_{X^{\prime}} \rightarrow 0 \quad \text { as } n \rightarrow+\infty .
$$
Since $\left(u_n\right)$ is bounded in $X$, it follows that $\left(\lambda_n\right)$ is also a bounded sequence.  We can assume, without loss of generality, that $\lambda_n \rightarrow \lambda_\epsilon$ as $n \rightarrow+\infty$. Hence,
$$
\left\|I_\epsilon^{\prime}\left(u_n\right)-\lambda_\epsilon \Psi^{\prime}\left(u_n\right)\right\|_{X^{\prime}} \rightarrow 0 \,\,\text { as } \,\, n \rightarrow+\infty
$$
and so,
$$
I_\epsilon^{\prime}\left(u_\epsilon\right)-\lambda_\epsilon \Psi^{\prime}\left(u_\epsilon\right)=0 \quad \text {in} \quad X^{\prime}.
$$
 The Lemma \ref{Brezis-Lieb} implies that
$$
I_\epsilon^{\prime}\left(u_n\right)u_n=I_\epsilon^{\prime}\left(u_\epsilon\right)u_\epsilon+I_\epsilon^{\prime}\left(v_n\right)v_n+o_n(1),
$$
and by applying Br\'{e}zis-Lieb Lemma (see \cite{Willem}), we obtain
$$
\Psi_\epsilon^{\prime}\left(u_n\right)u_n=\Psi_\epsilon^{\prime}\left(u_\epsilon\right)u_\epsilon+\Psi_\epsilon^{\prime}\left(v_n\right)v_n+o_n(1).
$$
From this,
$$
I_\epsilon^{\prime}\left(u_n\right)u_n-\lambda_\epsilon \Psi^{\prime}\left(u_n\right)u_n=I_\epsilon^{\prime}\left(u_\epsilon\right)u_\epsilon+I_\epsilon^{\prime}\left(v_n\right)v_n-\lambda_\epsilon \Psi^{\prime}\left(u_\epsilon\right)u_\epsilon-\lambda_\epsilon \Psi^{\prime}\left(v_n\right)v_n +o_n(1),
$$
from  where it follows that
$$
I_\epsilon^{\prime}\left(v_n\right)v_n-\lambda_\epsilon \Psi^{\prime}\left(v_n\right)v_n \rightarrow 0 \quad \text { as } \quad n \rightarrow+\infty,
$$
or equivalently,
\begin{equation} \label{NEWEQUATION}
	\int_{\mathbb{R}^N}\left(\left|\nabla v_n\right|^2+(V(\epsilon x)+1)\left|v_n\right|^2\right) d x+\int_{\mathbb{R}^N} F'_1(v_n)v_nd x-\lambda_\epsilon \int_{\mathbb{R}^N}\left|v_n\right|^2 d x=\int_{\mathbb{R}^N} F'_2(v_n)v_n d x+o_n(1).
\end{equation}
Then,
$$
0>\rho_1+\Upsilon_{0, a} \geq \liminf _{n \rightarrow+\infty} I_\epsilon\left(u_n\right)=\liminf _{n \rightarrow+\infty}\left(I_\epsilon\left(u_n\right)-\frac{1}{2} I_\epsilon^{\prime}\left(u_n\right) u_n+\frac{1}{2}\lambda_\epsilon a^2\right).
$$
Since
$$
I_\epsilon\left(u_n\right)-\frac{1}{2} I_\epsilon^{\prime}\left(u_n\right) u_n=\frac{1}{2}\int_{\mathbb{R}^N}|u_n|^2\,dx \geq 0,
$$
we conclude that
$$
0>\rho_1+\Upsilon_{0, a} \geq \frac{\lambda_\epsilon a^2}{2},
$$
and so,
$$
\limsup _{\epsilon \rightarrow 0} \lambda_\epsilon \leq \frac{2(\rho_1+\Upsilon_{0, a})}{a^2}<0.
$$
Therefore, there exists $\lambda_*<0$ independent of $\epsilon$ such that
$$
\lambda_\epsilon \leq \lambda_*<0, \quad \forall \epsilon \in\left(0, \epsilon_0\right).
$$
Based on the above analysis, the fact that $F_1^{\prime}(s) s \geq 0$  and the equality \eqref{NEWEQUATION} ensure that
$$
\int_{\mathbb{R}^N}\left(\left|\nabla v_n\right|^2+(V(\epsilon x)+1)\left|v_n\right|^2\right) d x-\lambda_* \int_{\mathbb{R}^N}\left|v_n\right|^2 d x \leq \int_{\mathbb{R}^N} F'_2(v_n)v_n d x+o_n(1).
$$
This together with the growth of $F_2$ yields
$$
\int_{\mathbb{R}^N}\left(\left|\nabla v_n\right|^2-\lambda_*\left|v_n\right|^2\right) d x \leq C_2\left|v_n\right|_p^p+o_n(1).
$$
 By using the continuous Sobolev embedding $H^{1}(\mathbb{R}^N) \hookrightarrow L^p\left(\mathbb{R}^N\right)$, we obtain
\begin{equation}\label{ine41}
\left\|v_n\right\|_{H^{1}}^2 \leq C_3\left|v_n\right|_p^p+o_n(1) \leq C_4\left\|v_n\right\|_{H^{1}}^p+o_n(1)
\end{equation}
where $C_3, C_4>0$ are independent of $\epsilon$. Since $v_n \nrightarrow 0$ in $ X$, for some subsequence of $\left(v_n\right)$, still denoted by $\left(v_n\right)$, we can assume that $\displaystyle \liminf_{n \rightarrow+\infty}\left\|v_n\right\|_{H^1}>0$, because if $v_n \to 0$ in $H^{1}(\mathbb{R}^N)$, we will have $F_2'(v_n)v_n \to 0$ in $L^{1}(\mathbb{R}^N)$.  This together with \eqref{NEWEQUATION} and the fact that $F_1'(s)s \geq F_1(s)$ for all $s \geq 0$ implies that $F_1(v_n) \to 0$ in $L^{1}(\mathbb{R}^N)$. As $F_1 \in (\Delta_2)$, the last limit yields that $v_n \to 0$ in $L^{F_1}(\mathbb{R}^N)$, then $v_n \to 0$ in $X$, which is absurd. Then, the  inequality \eqref{ine41} ensures that
\begin{equation}\label{ine42}
\liminf _{n \rightarrow+\infty}\left\|v_n\right\|_{H^{1}} \geq\left(\frac{1}{C_4}\right)^{\frac{1}{p-2}}.
\end{equation}
From \eqref{ine41} and \eqref{ine42},
\begin{equation}\label{ine43}
\liminf _{n \rightarrow+\infty}\left|v_n\right|_p^p \geq C_5,
\end{equation}
for some $C_5>0$ that does not depend on $\epsilon$. By employing the Gagliardo-Nirenberg inequality,
$$
\left|v_n\right|_p^p \leq C\left|v_n\right|_2^{\left(1-\beta_p\right) p}\left|\nabla v_n\right|_2^{\beta_p p} \quad \text { with } \quad \beta_p=N\left(\frac{1}{2}-\frac{1}{p}\right)
$$
we can further deduce that
\begin{equation}\label{ine44}
\liminf _{n \rightarrow+\infty}\left|v_n\right|_p^p \leq C\left(\liminf _{n \rightarrow+\infty}\left|v_n\right|_2\right)^{\left(1-\beta_p\right) p} K^{\beta_p p} \quad \text { with } \quad \beta_p=N\left(\frac{1}{2}-\frac{1}{p}\right)
\end{equation}
where $K>0$ is a constant independent of $\epsilon \in\left(0, \epsilon_0\right)$ satisfying $\left\|v_n\right\|_{H^1} \leq K$ for all $n \in \mathbb{N}$. Now the lemma follows from \eqref{ine43} and \eqref{ine44}.

\end{proof}
From now on, we fix $0<\rho<\min \left\{\frac{1}{2}, \frac{\beta}{a^2}\right\}\left(\Upsilon_{\infty, a}-\Upsilon_{0, a}\right) \leq \rho_1$, and provide a range of levels in which the functional $I_\epsilon$ satisfies the $(PS)$ condition constrained  to $S(a)$.

\begin{lemma}\label{compactness1234}
 For each $\epsilon \in\left(0, \epsilon_0\right)$, the functional $I_\epsilon$ satisfies the $(P S)_c$ condition constrained  to $S(a)$ for $c<\Upsilon_{0, a}+\rho$.
\end{lemma}
\begin{proof} Let $\left(u_n\right)$ be a $(P S)_c$ sequence for $I_\epsilon$ constrained  to $S(a)$ with $u_n \rightharpoonup u_\epsilon$ in $X_{0}$ and $c<\Upsilon_{0, a}+\rho$. Setting  the functional $\Psi: X \rightarrow \mathbb{R}$ given by
$$
\Psi(u)=\frac{1}{2} \int_{\mathbb{R}^N}|u|^2 d x,
$$
it follows that $S(a)=\Psi^{-1}\left(\left\{a^2 / 2\right\}\right)$. Then, by Willem \cite[Proposition~5.12]{Willem}, there exists $\left(\lambda_n\right) \subset \mathbb{R}$ such that
$$
\left\|I_\epsilon^{\prime}\left(u_n\right)-\lambda_n \Psi^{\prime}\left(u_n\right)\right\|_{X^{\prime}} \rightarrow 0 \text { as } n \rightarrow+\infty \text {. }
$$
Using Lemma \ref{nontrivial1}, if $v_n=u_n-u_\epsilon \not\rightarrow 0$ in $X$, there exists $\beta>0$ independent of $\epsilon$ such that
$$
\liminf _{n \rightarrow+\infty}\left|v_n\right|_2^2 \geq \beta
$$
Fix $d_n=\left|v_n\right|_2$ and assume that $\left|v_n\right|_2 \rightarrow d>0$ and $\left|u_\epsilon\right|_2=b$. Consequently,  $a^2=b^2+d^2$,  $b>0$ and $d_n \in(0, a)$ for $n\in \mathbb{N}$ large enough. Recalling the arguments explored in the proof of Lemma \ref{nontrivial}, we have that $I_\epsilon\left(v_n\right) \geq \Upsilon_{\infty, d_n}+o_n(1)$, from where it follows that
$$
c+o_n(1)=I_\epsilon\left(u_n\right)=I_\epsilon\left(v_n\right)+I_\epsilon\left(u_\epsilon\right)+o_n(1) \geq \Upsilon_{\infty, d_n}+\Upsilon_{0, b}+o_n(1) .
$$
Arguing as in the proof Lemma \ref{subadditive},
$$
\rho+\Upsilon_{0, a} \geq \frac{d_n^2}{a^2} \Upsilon_{\infty, a}+\frac{b^2}{a^2} \Upsilon_{0, a}.
$$
Letting $n \rightarrow+\infty$, we obtain
$$
\rho \geq \frac{d^2}{a^2}\left(\Upsilon_{\infty, a}-\Upsilon_{0, a}\right) \geq \frac{\beta}{a^2}\left(\Upsilon_{\infty, a}-\Upsilon_{0, a}\right),
$$
which is absurd, because $\rho<\frac{\beta}{a^2}\left(\Upsilon_{\infty, a}-\Upsilon_{0, a}\right)$. From $v_n \rightarrow 0$ in $X$, that is, $u_n \rightarrow u_\epsilon$ in $X$, implying that $\left|u_\epsilon\right|_2=a$ and
$$
-\Delta u_\epsilon+V(\epsilon x) u_\epsilon=\lambda_\epsilon u_\epsilon+u_\epsilon \log u_\epsilon^2, \text { in } \mathbb{R}^N,
$$
where $\lambda_\epsilon$ is the limit of some subsequence of $\left(\lambda_n\right)$.
\end{proof}
\subsection{Multiplicity result for \eqref{11}}
Let $\delta>0$ be fixed, and let $w$ be a positive solution of the problem
\begin{align}\label{231}
 \left\{
\begin{aligned}
&-\Delta u=\lambda u+u \log u^2,
\quad
\hbox{in}\,\,\mathbb{R}^N,\\
&\int_{\mathbb{R}^{N}}|u|^{2}dx=a^{2},
\end{aligned}
\right.
\end{align}
with $I_0(w)=\Upsilon_{0, a}$. Let $\eta$ be a smooth nonincreasing cut-off function defined in $[0, \infty)$ such that $\eta(s)=1$ if $0 \leq s \leq \frac{\delta}{2}$ and $\eta(s)=0$ if $s \geq \delta$. For any $y \in M$, we define
$$
\begin{gathered}
\Psi_{\epsilon, y}(x)=\eta(|\epsilon x-y|) w((\epsilon x-y) / \epsilon), \\
\tilde{\Psi}_{\epsilon, y}(x)=a \frac{\Psi_{\epsilon, y}(x)}{\left|\Psi_{\epsilon, y}\right|_2},
\end{gathered}
$$
and $\Phi_\epsilon: M \rightarrow S(a)$ by $\Phi_\epsilon(y)=\tilde{\Psi}_{\epsilon, y}$. By construction, for any $y \in M$, the function $\Phi_\epsilon(y)$ has compact support.
\begin{lemma}\label{compactness12}
The function $\Phi_\epsilon$ has the following limit
$$
\lim _{\epsilon \rightarrow 0} I_\epsilon\left(\Phi_\epsilon(y)\right)=\Upsilon_{0, a}, \text { uniformly in } \,\, y \in M.
$$
\end{lemma}
\begin{proof}
Assume, by contradiction, that there exist $\delta_0>0$, $\left(y_n\right) \subset M$ with $y_n \to y$ and $\epsilon_n \rightarrow 0$ such that
$$
\left|I_{\epsilon_n}\left(\Phi_{\epsilon_n}\left(y_n\right)\right)-\Upsilon_{0, a}\right| \geq \delta_0, \quad \forall n \in \mathbb{N}.
$$
By applying the Lebesgue Dominated Convergence Theorem, we can show that
$$
\begin{aligned}
& \lim _{n \rightarrow+\infty} \int_{\mathbb{R}^N}\left|\Psi_{\epsilon_n, y_n}\right|^2 d x=\lim _{n \rightarrow+\infty} \int_{\mathbb{R}^N}\left|\eta\left(\epsilon_n z\right) w(z)\right|^2 d z=\int_{\mathbb{R}^N}|w(z)|^2 d z=a^2,\\
& \lim _{n \rightarrow+\infty} \int_{\mathbb{R}^N} F_1\left(\Phi_{\epsilon_n}\left(y_n\right)\right) d x=\lim _{n \rightarrow+\infty} \int_{\mathbb{R}^N} F_1\left(a \frac{\eta\left(\epsilon_n z\right) w(z)}{\mid \Psi_{\epsilon_n, y_n \mid 2}}\right) d z=\int_{\mathbb{R}^N} F_1(w(z)) dz, \\
& \lim _{n \rightarrow+\infty} \int_{\mathbb{R}^N} F_2\left(\Phi_{\epsilon_n}\left(y_n\right)\right) d x=\lim _{n \rightarrow+\infty} \int_{\mathbb{R}^N} F_2\left(a \frac{\eta\left(\epsilon_n z\right) w(z)}{\mid \Psi_{\epsilon_n, y_n \mid 2}}\right) d z=\int_{\mathbb{R}^N} F_2(w(z)) dz, \\
& \lim _{n \rightarrow+\infty} \int_{\mathbb{R}^N}\left|\nabla \Phi_{\epsilon_n}\left(y_n\right)\right|^2 d x=\lim _{n \rightarrow+\infty} \int_{\mathbb{R}^N} \frac{a^2}{\left|\Psi_{\epsilon_n, y}\right|_2^2}\left|\nabla\left(\eta\left(\epsilon_n z\right) w(z)\right)\right|^2 d z=\int_{\mathbb{R}^N}|\nabla w(z)|^2 dz,
\end{aligned}
$$
and
$$
\lim _{n \rightarrow+\infty} \int_{\mathbb{R}^N} V\left(\epsilon_n x\right)\left|\Phi_{\epsilon_n}\left(y_n\right)\right|^2 d x=\lim _{n \rightarrow+\infty} \int_{\mathbb{R}^N}\frac{ V\left(\epsilon_n z+y_n\right)a^2 }{\left|\Psi_{\epsilon_n, y}\right|_2^2}\left|\eta\left(\epsilon_n z\right) w(z)\right|^2 d z=V_0\int_{\mathbb{R}^N}|w(z)|^2 d z .
$$
Consequently,
$$
\lim _{n \rightarrow+\infty} I_{\epsilon_n}\left(\Phi_{\epsilon_n}\left(y_n\right)\right)=I_{0, a}(w)=\Upsilon_{0, a},
$$
which is absurd.
\end{proof}

Now we define the barycenter map. For any $\delta>0$, let $R=R(\delta)>0$ be such that $M_\delta \subset B_R(0)$. Let $\chi: \mathbb{R}^N \rightarrow \mathbb{R}^N$ be defined as $\chi(x)=x$ for $|x| \leq R$ and $\chi(x)=\frac{R x}{|x|}$ for $|x| \geq R$. Finally, let us consider $\beta_\epsilon: S(a) \rightarrow \mathbb{R}^N$ given by
$$
\beta_\epsilon(u)=\frac{\int_{\mathbb{R}^N} \chi(\epsilon x)|u|^2 d x}{a^2}.
$$
\begin{lemma}\label{nonempty}
 The function $\Phi_\epsilon$ has the following limit
$$
\lim _{\epsilon \rightarrow 0} \beta_\epsilon\left(\Phi_\epsilon(y)\right)=y, \text { uniformly in } y \in M
$$
\end{lemma}
\begin{proof}
We will prove this lemma by contradiction. Assume the contrary, i.e., there exist $\delta_0>0$, $\left(y_n\right) \subset M$ with $y_n \to y$ and $\epsilon_n \rightarrow 0$ such that
\begin{equation}\label{ine61}
\left|\beta_{\epsilon_n}\left(\Phi_{\epsilon_n}\left(y_n\right)\right)-y_n\right| \geq \delta_0, \quad \forall n \in \mathbb{N}.
\end{equation}
Using the definition of $\Phi_{\epsilon_n}\left(y_n\right)$ and $\beta_{\epsilon_n}$, we have the following equality
$$
\beta_{\epsilon_n}\left(\Phi_{\epsilon_n}\left(y_n\right)\right)=y_n+\frac{\int_{\mathbb{R}^N}\left(\chi\left(\epsilon_n z+y_n\right)-y_n\right)\left|\eta\left(\epsilon_n z\right) w(z)\right|^2 d z}{a^2} .
$$
Since $\left(y_n\right) \subset M \subset B_R(0)$, the Lebesgue Dominated Convergence Theorem implies that
$$
\left|\beta_{\epsilon_n}\left(\Phi_{\epsilon_n}\left(y_n\right)\right)-y_n\right| \rightarrow 0, \quad \text { as } n \rightarrow+\infty
$$
which contradicts \eqref{ine61}. Hence the lemma is proved.
\end{proof}
Now, we prove the following useful compactness result.
\begin{proposition}\label{com}
 Let $\epsilon_n \rightarrow 0$ and $\left(u_n\right) \subset S(a)$ with $I_{\epsilon_n}\left(u_n\right) \rightarrow \Upsilon_{0, a}$. Then, there is $\left(\tilde{y}_n\right) \subset \mathbb{R}^N$ such that $v_n(x)=u_n\left(x+\tilde{y}_n\right)$ has a convergent subsequence in $X$. Moreover, up to a subsequence, $y_n=\epsilon_n \tilde{y}_n \rightarrow y$ for some $y \in M$.
 \end{proposition}
\begin{proof}
 Firstly, we claim that there are $R_0, \tau>0$ and $\tilde{y}_n \in \mathbb{R}^N$ such that
$$
\int_{B_{R_0\left(\bar{y}_n\right)}}\left|u_n\right|^2 d x \geq \tau, \quad \forall n \in \mathbb{N} \text {. }
$$
Otherwise, by Lions Lemma  we would have $u_n \rightarrow 0$ in $L^p\left(\mathbb{R}^N\right)$ for all $p \in\left(2,2^*\right)$, which leads to $\int_{\mathbb{R}^N} F_2\left(u_n\right) d x \rightarrow 0$. Since $F_1 \geq 0$, we must have $\displaystyle \lim _{n \rightarrow+\infty} I_{\epsilon_n}\left(u_n\right) \geq 0$. However, this contradicts the fact that $\displaystyle \lim _{n \rightarrow+\infty} I_{\epsilon_n}\left(u_n\right)=\Upsilon_{0, a}<0$. Therefore, fixing $v_n(x)=u_n\left(x+\tilde{y}_n\right)$, there is $v \in X \backslash\{0\}$ such that, for some subsequence, $v_n \rightharpoonup v$ in $X$. Since $\left(v_n\right) \subset S(a)$ and $I_{\epsilon_n}\left(u_n\right) \geq I_0\left(u_n\right)=I_0\left(v_n\right) \geq \Upsilon_{0, a}$, it follows that $I_0\left(v_n\right) \rightarrow \Upsilon_{0, a}$. By Theorem \ref{compactness}, we have $v_n \rightarrow v$ in $X$ and $v \in S(a)$.

\begin{claim} \label{Claim 5.2.} $\left(y_n\right)$ is bounded.
\end{claim}

Indeed, if for some subsequence $\left|y_n\right| \rightarrow+\infty$, the following limit holds
$$
\Upsilon_{0, a}=\lim _{n \rightarrow+\infty}\Big(\frac{1}{2}\int_{\mathbb{R}^N}|\nabla v_n|^2 d x+\int_{\mathbb{R}^N}( V(\epsilon_n x+y_n)+1)|v_n|^2 d x+\int_{\mathbb{R}^N} F_1(v_n)d x-\int_{\mathbb{R}^N} F_2(v_n)d x\Big).
$$
This leads to
$$
\Upsilon_{0, a} \geq \frac{1}{2}\int_{\mathbb{R}^N}|\nabla v|^2 d x+\frac{1}{2}(V_{\infty} +1)\int_{\mathbb{R}^N}|v|^2 d x+\int_{\mathbb{R}^N} F_1(v)d x-\int_{\mathbb{R}^N} F_2(v)d x \geq \Upsilon_{\infty, a},
$$
which contradicts \eqref{comparison}.
Since Claim  \ref{Claim 5.2.} holds, we can assume that $y_n \rightarrow y$ in $\mathbb{R}^N$. Following a similar argument as above, we obtain
$$
\Upsilon_{0, a} \geq \frac{1}{2}\int_{\mathbb{R}^N}|\nabla v|^2 d x+\frac{1}{2}(V(y)+1)\int_{\mathbb{R}^N}|v|^2 d x+\int_{\mathbb{R}^N} F_1(v)d x-\int_{\mathbb{R}^N} F_2(v)d x  \geq \Upsilon_{V(y), a} .
$$
By Corollary \ref{coro1}, we have $\Upsilon_{V(y), a}>\Upsilon_{0, a}$ when $V(y)>V_0$. Since $V(y) \geq V_0$ for all $y \in \mathbb{R}^N$, the above inequality implies that $V(y)=V_0$, which means $y \in M$.
\end{proof}
Let $h:[0,+\infty) \rightarrow[0,+\infty)$ be a positive function such that $h(\epsilon) \rightarrow 0$ as $\epsilon \rightarrow 0$ and let
\begin{equation}\label{sublevel}
\tilde{S}(a)=\left\{u \in S(a): I_\epsilon(u) \leq \Upsilon_{0, a}+h(\epsilon)\right\}.
\end{equation}
Thanks to Lemma \ref{compactness12}, we can choose $h(\epsilon)=\sup _{y \in M}\left|I_\epsilon\left(\Phi_\epsilon(y)\right)-\Upsilon_{0, a}\right|$. Therefore, we have \linebreak $\Phi_\epsilon(y) \in \bar{S}(a)$ for all $y \in M$.\\

We have the following relation between $\tilde{S}(a)$ and the barycenter map $\beta_\epsilon(u)$.
\begin{lemma}\label{concentration}
 Let $\delta>0$ and $M_\delta=\left\{x \in \mathbb{R}^N: \operatorname{dist}(x, M) \leq \delta\right\}$. Then,
$$
\lim _{\epsilon \rightarrow 0} \sup _{u \in \bar{S}(a)} \inf _{z \in M_\delta}\left|\beta_\epsilon(u)-z\right|=0 .
$$
 \end{lemma}
\begin{proof}
Let $\epsilon_n \rightarrow 0$ and $u_n \in \tilde{S}(a)$ such that
$$
\inf _{z \in M_\delta}\left|\beta_{\epsilon_n}\left(u_n\right)-z\right|=\sup _{u \in \bar{S}(a)} \inf _{y \in M_\delta}\left|\beta_{\epsilon_n}\left(u_n\right)-z\right|+o_n(1) .
$$
From the above equality, it suffices to find a sequence $\left(y_n\right) \subset M_\delta$ such that
$$
\lim _{n \rightarrow+\infty}\left|\beta_\epsilon\left(u_n\right)-y_n\right|=0.
$$
Since $u_n \in \tilde{S}(a)$, we have
$$
\Upsilon_{0, a} \leq I_0\left(u_n\right) \leq I_{\epsilon_n}\left(u_n\right) \leq \Upsilon_{0, a}+h\left(\epsilon_n\right), \quad \forall n \in \mathbb{N},
$$
and so,
$$
u_n \in S(a)\, \text { and }\, I_{\epsilon_n}\left(u_n\right) \rightarrow \Upsilon_{0, a}.
$$
From Proposition \ref{com} , there exists $\left(\tilde{y}_n\right) \subset \mathbb{R}^N$ such that $y_n=\epsilon_n \tilde{y}_n \rightarrow y$ for some $y \in M$ and $v_n(x)=u_n\left(x+\tilde{y}_n\right)$ is strongly convergent to some $v \in X_{0}$ with $v \neq 0$. Consequently, $\left(y_n\right) \subset M_\delta$ for $n\in\mathbb{N}$ large enough and
$$
\beta_{\epsilon_n}\left(u_n\right)=y_n+\frac{\int_{\mathbb{R}^N}\left(\chi\left(\epsilon_n z+y_n\right)-y_n\right)\left|v_n\right|^2 d z}{a^2}.
$$
This implies
$$
\beta_{\epsilon_n}\left(u_n\right)-y_n=\frac{\int_{\mathbb{R}^N}\left(\chi\left(\epsilon_n z+y_n\right)-y_n\right)\left|v_n\right|^2 d z}{a^2} \rightarrow 0 \quad \text { as } n \rightarrow+\infty,
$$
which proves the lemma.
\end{proof}
\subsection{Proof of Theorem \ref{T1}}
We will divide the proof into two parts:\\

Part I: Multiplicity of solutions. Let $\epsilon \in\left(0, \epsilon_0\right)$. By applying Lemmas \ref{compactness12}, \ref{nonempty} and \ref{concentration}, we can follow the argument in \cite{CLa} to conclude that $\beta_\epsilon \circ \Phi_\epsilon$ is homotopic to the inclusion map $\text{id}:M \rightarrow M_\delta$. Thus, we have
$$
\operatorname{cat}(\tilde{S}(a)) \geq \operatorname{cat}_{M_\delta}(M)
$$
Using arguments similar to Lemma \ref{coercive}, we can also show that $I_\epsilon$ is bounded from below on $S(a)$. Since  the functional $I_\epsilon$ satisfies the $(P S)_c$ condition for $c \in\left(\Upsilon_{0, a}, \Upsilon_{0, a}+h(\epsilon)\right)$ by Lemma \ref{compactness1234}, we can apply the Lusternik-Schnirelman category of critical points (see \cite{Gh} and \cite{Willem}) to conclude that $J_\epsilon$ has at last $\operatorname{cat}_{M_\delta}(M)$ critical points on $S(a)$.\\

Part II: The behavior of maximum points of $\left|u_\epsilon\right|$. Let $u_\epsilon$ a solution of \eqref{11} with $I_\epsilon\left(u_\epsilon\right) \leq \Upsilon_{0, a}+h(\epsilon)$, where $h$ was given in \eqref{sublevel}. Following the proof of Proposition \ref{com}, for each $\epsilon_n \rightarrow 0$, there is $\tilde{y}_n \in \mathbb{R}^N$ such that $y_n=\epsilon_n \tilde{y}_n \rightarrow y$ with $y \in M$ and $v_n(x)=u_{\epsilon_n}\left(x+\tilde{y}_n\right)$ is strongly convergent to $v \in X$ with $v \neq 0$. As $v_n$ is a solution of
$$
-\Delta v_n+V\left(\epsilon_n x+y_n\right) v_n=\lambda_n v_n+v_n \log v_n^2, \,\,\text { in } \,\,\mathbb{R}^N,
$$
with
$$
\limsup _{\epsilon \rightarrow 0} \lambda_n \leq \frac{\rho_1+\Upsilon_{0, a}}{a^2}<0,
$$
the convergence $v_n \rightarrow v$ in $X$ permits to apply the same arguments found in  \cite[Lemma 4.5]{AFi}. This implies that
$$
\lim _{|x| \rightarrow+\infty} v_n(x)=0, \quad \text { uniformly in } \mathbb{N}
$$
which means that given $\tau>0$, there are $R_1>0$ and $n_0 \in \mathbb{N}$ such that
$$
\left|v_n(x)\right| \leq \tau \quad \text { for } \quad|x| \geq R_1 \quad \text { and } \quad n \geq n_0.
$$
We claim that $\left|v_n\right|_{\infty} \nrightarrow 0$. Otherwise, if $\left|v_n\right|_{\infty} \nrightarrow 0$,  we would have that $v_n \rightarrow 0$ in $X$, which is impossible since $\left|v_n\right|_2=a$ for all $n \in \mathbb{N}$. Now, let us fix $\tau>0$ such that $\left|v_n\right|_{\infty} \geq 2 \tau$ and consider $z_n \in \mathbb{R}^N$ satisfying $\left|v_n\left(z_n\right)\right|=\left|v_n\right|_{\infty}$ for all $n \in \mathbb{N}$. The  analysis above ensures that $\left|z_n\right| \leq R_1$ for all $n \in \mathbb{N}$. Furthermore, let $\xi_n \in \mathbb{R}^N$ such that $\left|u_n\left(\xi_n\right)\right|=\left|u_n\right|_{\infty}$ for all $n \in \mathbb{N}$. Then, we have $\xi_n=z_n+\tilde{y}_n$ and
$$
\lim _{n \rightarrow+\infty} V\left(\epsilon_n \xi_n\right)=\lim _{n \rightarrow+\infty} V\left(\epsilon_n z_n+\epsilon_n \tilde{y}_n\right)=V(y)=V_0
$$
which completes the proof.

\noindent \textsc{Claudianor O. Alves } \\
Unidade Acad\^{e}mica de Matem\'atica\\
Universidade Federal de Campina Grande \\
Campina Grande, PB, CEP:58429-900, Brazil \\
\texttt{coalves@mat.ufcg.edu.br} \\
\noindent and \\
\noindent \textsc{Chao Ji} \\
Department of Mathematics\\
East China University of Science and Technology \\
Shanghai 200237, PR China \\
\texttt{jichao@ecust.edu.cn}\\

\end{document}